\documentclass[reqno]{article}
\usepackage{amssymb, amsmath, amsthm, amsfonts, amscd}
\usepackage[english]{babel}

\usepackage{graphicx}

\usepackage{epstopdf}

\usepackage{breqn}
\usepackage{mathtools}
\usepackage{color}
\usepackage{hyperref}
\hypersetup{
    colorlinks=true,
    citecolor=red,
    filecolor=black,
    linkcolor=blue,
    urlcolor=black
}

\chardef\bslash=`\\ 

\newtheorem{theorem}{Theorem}[section]
\newtheorem{corollary}[theorem]{Corollary}
\newtheorem{lemma}[theorem]{Lemma}
\newtheorem{proposition}[theorem]{Proposition}
\newtheorem{definitionn}{Definition}[section]

\newtheorem{assumption}[theorem]{Assumption}

\newcommand{\N}{\mathbb{N}}

\newcommand{\HH}{\mathrm{H}}

\newcommand{\DD}{\mathcal{D}}

\newcommand{\ZZ}{\mathcal{Z}}

\newcommand{\Z}{\mathbb{Z}}
\newcommand{\Q}{\mathbb{Q}}
\newcommand{\R}{\mathbb{R}}
\newcommand{\C}{\mathbb{C}}

\newcommand{\T}{\mathbb{T}}
\newcommand{\Td}{\mathbb{T}^{d}}
\newcommand{\Rd}{\mathbb{R}^{d}}
\newcommand{\Zd}{\mathbb{Z}^{d}}
\newcommand{\TT}{\mathcal{T}}

\newcommand{\ft}{\mathfrak{t}}

\newcommand{\Id}{\mathrm{Id}}
\newcommand{\Conj}{\mathrm{Conj}}

\newcommand{\sm}{C^{\infty}}
\newcommand{\0}{\circ}

\def\a{\alpha }
\def\g{\gamma }
\def\sm{C^{\infty} }

\def\l{\lambda }
\def\L{\Lambda }

\def\r{\rho}
\def\s{\sigma}
\def\t{\tau}
\def\D{\Delta}

\def\w{\omega}

\def\.{\cdot }
\def\ra{\rightarrow}

\def\begeq{\begin{equation*}}
\def\endeq{\end{equation*}}

\title{
\textsc{\textbf{Fibered rotation vector and hypoellipticity for quasiperiodic
cocycles in compact Lie groups}}\\
\author{N. Karaliolios}
\author{Nikolaos Karaliolios \footnote{Universit\'{e} de Lille. Email: nikolaos.karaliolios@univ-lille.fr}}
}

\begin{document}

\maketitle

\begin{abstract}
Using weak solutions to the conjugation equation, we define a fibered rotation
vector for almost reducible quasi-periodic cocycles in $\mathbb{T}^{d} \times G$, $G$ a
compact Lie group, over a Diophantine rotation. We then prove that if this rotation
vector is Diophantine with respect to the rotation in $\mathbb{T}^{d}$, the cocycle is
smoothly reducible, thus establishing a hypoellipticity property in the spirit of the
Greenfield-Wallach conjecture in PDEs.


\end{abstract}  

\tableofcontents

\section{Introduction} \label{secintro}

In this note, we are interested in the regularity of the solutions $\HH$ to the conjugation
equation for dynamical systems,
\begin{equation} \label{eqconj}
\HH \0 \Phi _{1}= \Phi _{2} \0 \HH  
\end{equation}
where the $\Phi _{i}$ are two $\sm$ smooth diffeomoprhisms of a compact manifold
$\mathcal{M}$ and $\HH$ can be a smooth or finitely differentiable diffeomorphism, a
homeomorphism, or a measurable self-mapping well defined and invertible $a.e.$. We are
interested in particular in definining weak solutions $\HH $ in $\DD '$ and then
imposing a condition on, say, $\Phi _{2}$ which guarantees $\sm$ smoothness for $\HH$.

In our case, the manifold $\mathcal{M}$ will be the product space $\Td \times G$, with
$\Td = \Rd / \Zd$ and $G$ a compact Lie group. We will not consider the full space of
diffeomorphisms of this space, but work in the space of \textit{smooth quasi-periodic
cocycles}, denoted by $SW^{\infty} = SW^{\infty}(\Td ,G)$. A cocycle
$\Phi = (\a , A(\. )) \in SW^{\infty}$, where $A(\. ) \in \sm (\Td ,G) $, acts on
$\Td \times G$ by
\begeq
\Phi .(x,S) \mapsto (x+\a , A(x).S)
\endeq
The space $SW^{\infty}$ has the natural product topology, and is thus a Fr\'{e}chet space.
The dynamics of a cocycle fiber over $\Td$. We define the projection $\pi$ of the cocycle
on its \textit{frequency}, $\pi (\Phi ) = \a$, and introduce the notation
\begeq
SW_{\a}^{\infty}(\Td ,G) = \pi ^{-1}(\{ \a \})
\endeq
As the rotation $\a$ will be fixed throughout the article, we take the liberty to make
no distinction between the mapping $A(\. ) \in \sm (\Td ,G)$ and the cocycle over $\a$ that
it defines.

The relevant notion of conjugation in $SW^{\infty}$ is \textit{fibered conjugation},
via the action of the subgroup $SW^{\infty}_{0}(\Td ,G)$. The resulting conjugation
equation, since conjugation does not act on frequences, is
\begeq
A_{1}(\. ) = H(\. + \a).A_{2}(\.) H^{*}(\. )
\endeq
where we suppose that the cocycle $\HH = (0,H(\. ))$ conjugates the cocycles
$\Phi _{j}= (\a , A_{j}(\. )) $, and the unknown is the mapping $H(\. )$.
We will use the notation $\Phi _{1} = \Conj _{H(\. )}\Phi _{2} = \Conj _{\HH}\Phi _{2}$.
In the particular case where $\Phi _{2}$ is a \textit{constant cocycle}, i.e. where
$A_{2}(\. ) \equiv A_{2} \in G$ is a constant mapping, we say that $\Phi _{1}$ is
\textit{reducible}.

In our previous paper \cite{NKRigidity} we showed that a cocycle in an open subset of
$SW_{\a}^{\infty}$ with $\a$ Diophantine, is $\sm$ reducible provided that it is measurably
conjugate to a constant cocycle $(\a , A_{\infty})$, $A_{\infty} \in G$, whose eigenvalues
$a_{\infty}$ satisfy a Diophantine condition with respect to $\a$.
Our subsequent work, \cite{NKInvDist} and \cite{NKContSpec}, shows that the Diophantine
condition on $A_{\infty}$ is optimal for this rigidity property.

In this note, we rid the main theorem of \cite{NKRigidity} of the ad-hoc hypothesis
of the existence a measurable conjugation, by showing that any cocycle in that open
subset can be assigned a fibered rotation vector.
Subsequently, we show that any cocycle is reducible to a constant one with the same
rotation vector, but by a transfer function with negative and finite Sobolev regularity
(depending only on the dimension of the phase-space).
We then use the hypoellipticity property implied by the same Diophantine hypothesis as in
\cite{NKRigidity} in order to obtain a $\sm$ smooth conjugation.

The hypothesis of the existence of a measurable conjugation is, thus, totally
unnatural, since for general reasons a distributional solution to the conjugation
equation exists and it is sufficient in order to obtain the existence of a $\sm$ smooth
solution. We remark that, by \cite{NKPhD},
when $\a \in RCD \subset \T ^{1}$, almost reducible cocycles form an open dense set in the
space of cocycles\footnote{The group $G$ needs to be semisimple for that to hold.}, and that cocycles that are not almost reducible do not have a fibered
rotation vector, because arbitrarily small perturbations of such cocycles admit the
skew-shift mapping of $\T ^{2}$
\begin{equation} \label{eqskew shift}
(x,y) \mapsto (x+\a , y+x)
\end{equation}
as a factor.\footnote{Actually, use of the K.A.M. normal form and of the sequence of
conjugations $H_{i}(\. )$, defined in eq. \eqref{eqseq conj}, in \S 10 of \cite{NKPhD} can
show that such cocycles themselves admit the skew-shift as a factor. This technical issue
is beyond the scope of the present note.}

The results that we prove herein admit their continuous-time counter parts, where
the corresponding objects are vector fields
\begeq
\begin{pmatrix}
\w \\
F(\. )
\end{pmatrix}
\in \Td \times g , F(\. ) : \Td \ra g
\endeq
with $g$ a compact Lie algebra, which places the results directly in the context of
hypoelliptic first order differential operators and the Greenfield-Wallach conjecture
(problem 2 in \cite{GrWal73}).

\section{Notation}

In order to keep this note short, we will follow the notation of our previous works that
are cited herein, and whenever some deviaton from that notation is needed, it will be
made explicit.

\subsection{Lie groups}

We denote by $G$ a compact Lie group, by $g$ its Lie algebra, and by $d'$ its dimension.
We fix $\TT$, a maximal torus of $G$ and denote by $\ft $ the Lie albebra of $\TT$.
We will denote by $\ZZ$ the lattice of preimages of the center of $G$ in $\ft$.

The adjoint action of $\TT $ on $g $ admits an eigenspace decomposition, known as
root space decomposition (cf. \cite{DieudElV}; we will follow the notation we used in
\cite{NKPhD} and \cite{NKRigidity}). There exists a finite set $\D \subset \ft ^{*}$ of
non-zero $\R$-valued linear forms on $\ft$ and vectors $j_{\r} \in \ft ^{\perp}$, for
$\r \in \D$ for which the following holds.

If $A \in \TT $ and $a \in \ft$ is such that $A = \exp (a)$, then
\begeq
Ad(A).j_{\r} = 2i\pi \r(a) j_{\r}, \forall \r \in \D
\endeq
Moreover, there exists a vector $e _{\r} \in \ft$ such that
\begeq
\R e_{\r} \oplus \C j_{\r} \equiv \R \times \C \approx (\R ^{3} , \. , \times )
\approx (su(2))_{\r}
\endeq
is naturally isomorphic to $\R ^{3}$ equipped with its scalar and vector product, and the
Lie bracket of $g$ restricted to $\R e_{\r} \oplus \C j_{\r}$ is mapped to the vector
product of $\R ^{3}$. This Lie algebra is naturally isomorphic to $su(2)$, the Lie algebra
of $SU(2)$, the group of special $2\times 2$ unitary matrices.

The vectors $e _{\r}$ spam $\ft$, and $A \in Z _{G}$, the centre of $G$, iff
\begeq
\exp (2i\pi \r (a) )  = 1 , \forall \r \in \D
\endeq
This condition defines a lattice $\ZZ \subset \ft $ which is fixed along with $\ft$.

\subsection{Functional analysis}

We will denote by $\| \. \| _{s}$ the norm of the Sobolev space
$H^{s} = H^{s} (\Td ,G)$. The space of
distributions will be denoted by
\begeq
\DD ' = \DD ' (\Td ,G) = \cup _{s \in \R}H^{s} (\Td ,G)
\endeq

The following definition is from \cite{GrWal73}.

\begin{definitionn}
A differential operator $P$ is called Globally Hypoelliptic if
\begeq
P.v = f
\endeq
with $v \in \DD '$ and $f \in \sm$ implies that $v \in \sm$.
\end{definitionn}

In our context, the differential operator is replaced by the conjugation operator, the
operator that defines the conjugation equation, eq. \eqref{eqconj}. This operator
behaves like a differential operator, and its equivalent in the continuous time case,
i.e. in the study of vector fields, is actually the Lie derivative with respect to a
vector field.

\subsection{Arithmetics}

\begin{definitionn} \label{defDC}
We will denote by $DC(\g ,\t )$ the set of numbers $\a$ in $\T ^{d} $
such that for any $k\in \Z ^{d} \setminus \{ 0 \}$,
\begin{equation*}
|\a \. k| _{\Z }\geq \frac{\gamma ^{-1}}{|k|^{\tau }}
\end{equation*}
Numbers satisfying such a condition are called Diophantine.
\end{definitionn}
The set $DC(\gamma ,\tau )$, for $\tau >d+1$ fixed and $\gamma \in \R _{+} ^{\ast}$ is
of positive Haar measure in $\T ^{d}$. If we fix $\tau$ and let $\gamma$ run through
the positive real numbers, we obtain $CD(\t ) =\cup_{\gamma >0} DC(\gamma ,\tau )$
which is of full Lebesgue measure. The numbers that do not satisfy any Diophantine
condition are called \textit{Liouvillean}. They form a residual set of $0$
Lebesgue measure.

We now turn to definitions concerning arithmetics relative to a rotation $\a$.
\begin{definitionn} \label{defrelDC}
We will denote by $DC_{\a }(\tilde{\g },\tilde{\t} )$ the set of numbers $\beta$ in $\T $
such that for any $k\in \Z ^{d} \setminus \{ 0 \}$,
\begin{equation*}
|\beta - \a \. k| _{\Z }\geq \frac{\tilde{\g } ^{-1}}{|k|^{\tilde{\t } }}
\end{equation*}
Such numbers are called Diophantine with respect to $\a$.
\end{definitionn}
In particular, $\a \in DC$ iff $0 \in DC_{\a}$.
\begin{definitionn} \label{defres}
We will denote by $Res_{\a }$ the set of numbers $\beta$ in $\R $
for which there exists $k_{r }\in \Zd \setminus \{ 0 \}$ such that
\begeq
\beta = k_{r} \. \a \mod \Z
\endeq
Such numbers are called Resonant with respect to $\a$.
\end{definitionn}
We end this block of definitions with some definitions concerning arithmetics of
elements of $G$.
\begin{definitionn} \label{defDC group}
We will denote by $DC_{\a}$ the set of elements $S \in G$ whose roots
are Diophantine with respect to $\a$.
\end{definitionn}
\begin{definitionn} \label{defres group}
We will denote by $Res_{\a }$, the set of elements $S \in G$ having at least one
root that is resonant with respect to $\a$.
\end{definitionn}
We will also use the terms Resonant or Diophantine roots, by obvious extension of the
respective concepts.

We finally recall the notion of Recurrent Diophantine numbers, defined only when $d=1$.
\begin{definitionn} \label{defRDC}
We will denote by $RDC(\gamma ,\tau )$ the set of  \textit{recurrent Diophantine}
numbers, i.e. the $\a $ in $\T \setminus \Q$ such that
$\mathrm{G}^{n}(\a)\in DC(\gamma ,\tau )$ for infinitely many $n$.
\end{definitionn}
In purely local notation, $\mathrm{G}(\a)=\{\a^{-1}\}$ is the Gauss map ($\{ \. \} $ stands
for "fractional part"). The set $RDC$ is also of full measure.

\subsection{Cocycles}

The dynamics and the notion of conjugation having already been introduced in \S
\ref{secintro}, we give directly the following definition on conjugation.

\begin{definitionn}
Two cocycles $\Phi _{j} \in SW^{\infty}_{\a}(\Td , G)$, $j=1,2$, are $H^{s}$ conjugate,
with $s \in \R$, iff there exists a sequence $H_{i}(\. ) \in \sm (\Td , G)$ of
conjugations such that
\begeq
\Conj _{H_{i} (\.)}\Phi _{1} \overset{\sm }{\ra} \Phi _{2}
\text{ and }
H_{i} (\.) \overset{H^{s}}{\ra} H (\.) 
\endeq
\end{definitionn}

This definition coincides with the classical one when $s\geq 0$, but is also meaningful
for $s<0$.

\begin{definitionn}
A cocycle $\Phi \in SW^{\infty}_{\a}(\Td , G)$, is $H^{s}$ reducible, with $s \in \R$,
iff there exists a sequence $H_{i}(\. ) \in \sm (\Td , G)$ of conjugations and a
constant $A_{\infty}$ such that
\begeq
\Conj _{H_{i} (\.)}\Phi _{1} \overset{\sm }{\ra} (\a , A_{\infty})
\text{ and }
H_{i} (\.) \overset{H^{s}}{\ra} H (\.) 
\endeq
\end{definitionn}

Following \cite{El2002a}, we define Almost Reducibility.
\begin{definitionn} \label{defar}
A cocycle $\Phi \in SW^{\infty}_{\a}(\Td , G)$, is almost reducible, iff there exists
a sequence $H_{i}(\. ) \in \sm (\Td , G)$ of conjugations and a sequence of constant
cocycles $\Phi _{i}$ such that
\begeq
(\Phi _{i})^{-1} \0 \Conj _{H_{i}(\. )}\Phi \overset{\sm }{\ra} \Id \in SW^{\infty}
(\Td ,G)
\endeq
\end{definitionn}
Almost Reducibility thus amounts, with some obvious notation, to
\begeq
A_{i}^{*}.H_{i}(\. +\a ).A(\. ).H_{i}^{*}(\.) \overset{\sm }{\ra} \Id \in G
\endeq

The first result of the present paper is the justification of the following definition.
\begin{definitionn} \label{defrot vec}
Let $\a \in DC$ and $\Phi \in SW^{\infty}_{\a}$ be almost reducible. Let also
$H (\.) \in H^{s}$ for some $s \in \R$ and let $A _{\infty} = \exp (a_{\infty}) \in \TT$ with
$a _{\infty} \in \ft$ be such that
\begeq
\Phi = \Conj _{\HH } (\a , A _{\infty})
\endeq
Then, we will say that $a _{\infty}$ is a rotation vector for $\Phi $ and write
$a _{\infty} = \varrho (\Phi )$.\footnote{We point out the trouble in notation. We use
$\r$ for the roots of an element of $G$ and $\varrho $ for the rotation vector. Both
quantities are customarily denoted by $\r$, and we apologize for this notation.}
\end{definitionn}

We also point the reader to \S 6, theorem 6.1 of \cite{NKCohomRigAK} for the following
statement considering Almost Reducible cocycles.

\begin{theorem} [\cite{El2002a}, \cite{KrikAst}, \cite{NKPhD}] \label{thm almost red}
Let $G$ be a compact Lie group, and $\a \in DC (\g , \t)$. Then, Almost Reducibility
holds in an open subset of $SW^{\infty} (\Td ,G)$.

More precisely, the K.A.M. scheme that proves Almost Reducibility produces:
\begin{enumerate}
\item a fast increasing sequence $N_{n} \in \N ^{*}$, $N_{n+1} = N_{n}^{1+ \s }$ for
some $0 < \s < 1$
\item a sequence of constants $A _{n} \in G$
\item a number $M \in \N \cup \{ \infty \}$ of resonant steps, a number $\nu > \t $,
and a subsequence $n_{i}$ of resonant steps where $A_{n_{i}} \notin DC_{\a}(1,\nu ) $,
$1 \leq i < M$
\item a sequence of resonant constants
$\{ \L _{i}  \} _{i =1}^{M} \subset Res _{\a} \subset G$
whose resonance $k_{n_{i}}^{\r} \in \Zd \setminus
\{0\}$ in each resonant subalgebra $(su(2))_{\r}$, satisfies
$|k_{n_{i}}^{\r} | \leq N_{n_{i}}$, as well as the following properties.
The constant $\L _{i}$ commutes with $A_{n_{i}}$ and
\begeq
d(\L _{i}  , A_{n_{i}} ) <
N_{n_{i}} ^{-\nu }
\endeq
if such an $\L _{n_{i}}$ exists.
The number $M \in \N \cup \{ \infty \} $ counts the number of steps of the K.A.M. scheme
for which $\L _{n_{i}}$ is defined
\item \label{item exp dec Y} a sequence of conjugations
$\exp (Y_{n}(\. ) )\in \sm (\T ^{d} ,G)$ satisfying
\begeq
\| Y_{n}(\. ) \|_{s} = O( N_{n}^{-\infty }), \forall s \geq 0
\endeq
\item \label{item pol growth B} a sequence of torus morphisms
$B_{n_{i}}(\. ) \in \sm (\Td ,G)$ satisfying
\begeq
B_{n_{i}} (\Zd ) \in \ZZ
\text{ and }
\| B_{n_{i}}(\. ) \|_{s} \simeq C_{s} N_{n_{i}}^{s+ \l }, \forall s \geq 0
\endeq
for some constant $\l >0$. The $B_{n_{i}}(\.)$ commute with the respective $A_{n_{i}}$ and
$\L _{n_{i}}$, and the constant
\begeq
B_{n_{i}}(\. + \a ). \L _{n_{i}} .B_{n_{i}}^{-1}(\. )
\endeq
is $N_{n_{i}}^{-\t }$-away from resonant constants. If $A_{n} $ is $N_{n}^{-\t }$-away
from resonant constants (i.e. if $n$ is not a resonant step) then $B_{n}(\. )$ is by
convention defined as $\equiv \Id$
\item a sequence of mappings $F_{n}(\. ) \in \sm (\T ^{d}, g)$ satisfying
\begeq
\| F_{n}(\. ) \|_{s} = O( N_{n}^{- \infty }), \forall s \geq 0
\endeq
\end{enumerate}
and such that the conjugation constructed iteratively following $H_{0} = \Id$ and
\begeq
H_{n} (\. )= \begin{cases}
e^{Y_{n}(\. )} H_{n-1} (\. ) \text{ if } n \notin \{ n_{i} \} \\
B_{n}(\. ) e^{Y_{n}(\. )} H_{n-1}(\. ) \text{ if } n \in \{ n_{i} \}
\end{cases}
\endeq
satisfies
\begeq
H_{n} (\. + \a ) Ae^{F(\. )} H_{n}^{-1}(\. ) = A_{n}. e^{F_{n}(\. )}
\endeq
\end{theorem}

In this notation, the K.A.M. normal form of almost reducible cocycles (see
\cite{NKRigidity}, \cite{NKInvDist} and \cite{NKContSpec} for a figure) is established
as follows.
\begin{lemma}\label{lemKAM normal form}
Under the assumptions of theorem \ref{thm almost red}, there exists a conjugation
$D(\. ) \in \sm (\Td ,G)$ such that the K.A.M. scheme applied to the cocycle
$\Conj _{D(\. )} \Phi $ produces only resonant steps.
\end{lemma}
For a cocycle in normal form, we abbreviate $n_{i}$ to $i$. For a cocycle in normal form,
all close-to-the-$\Id$ conjugations of theorem \ref{thm almost red} are
$\exp (Y_{n}(\. ))\equiv \Id$. The conjugations $B_{i}(\. )$ are\footnote{Essentially, but
the difference is cumbersome to write down and irrelevant to the goal of the note. The
reader can consult \S 9 of \cite{NKPhD} for the details.} of the form
\begin{equation} \label{eqfar from id conj}
B_{i}(x ) = 
\exp( D_{i}).
\exp
(
\sum _{\r} 2i\pi (k^{\r}_{i}\. x )e _{\r}
)
\exp(-D_{i})
\end{equation}
where $\exp(D_{i})$ conjugates a maximal torus passing by $A_{i}$ to a maximal torus
passing by $A_{i+1}$.

We point out that if the K.A.M. normal form is finite, i.e. if $M<\infty$, then $\Phi$
is $\sm$ reducible, since the sequence $H_{n}(\. )$ converges in $\sm$.

\section{Statement of results}

\begin{assumption}
In all following theorems, $\a$, the frequency of the cocycle $\Phi = (\a , A(\.))$, is
Diophantine. The cocycle $\Phi \in SW^{\infty}_{\a}(\Td ,G)$ is assumed to be
almost reducible.
\end{assumption}


Our first theorem implies that definition \ref{defrot vec} is not void.

\begin{theorem} \label{thmxrot vec}
For each cocycle $\Phi $ there exists a fibered rotation vector $\varrho = \varrho (\Phi )$
satisfying definition \ref{defrot vec}.
\end{theorem}

Concerning the invariance of the rotation vector under conjugations, we prove
the following propositions.

\begin{proposition} \label{propxrot vec transl}
The fibered rotation vector $\varrho $ is invariant under conjugation in
$\DD '$, and well defined in $\ft \mod (\a \. \Zd)\ZZ$.
\end{proposition}
We remark that the same phenomenon of indeterminacy occurs for rotation vectors of
diffeomorphisms of the torus $\Td$, $d \geq 2$, when we allow non homotopic to the $\Id $
conjugations to act.

\begin{proposition} \label{propxrot vec inv}
The fibered rotation vector $\varrho$ is invariant under conjugations $B(\. )$ for which
there exists $b(\. ) : \Td \ra g$ such that $B(\. ) = \exp (b(\. ))$. Representatives
of $\varrho (\Phi )$ depend continuously on $\Phi $ in the $C^{0}$ topology for fixed
$\pi (\Phi )$.
\end{proposition}
We remark that a mapping $\Td \ra G$ may lift to a mapping $\Rd \ra g$, for example the
conjugations $B_{i} (\. )$ of theorem \ref{thm almost red}.
Conjugations in the space of cocycles
that admid a bounded lift in $g$ play the role of homotopic-to-the-$\Id$ conjugations
for torus diffeomorphisms.

\begin{theorem}  \label{thmxred in neg reg}
The fibered rotation vector $\varrho$ is a total invariant for $H^{-d-d'}$ conjugation:
every almost reducible cocycle is $H^{-d-d'}$ conjugate to a constant one with the
corresponding rotation vector.
\end{theorem}

We finally prove the following hypoellipticity theorem, stating that under some
relevant arithmetic assumptions, the conjugation of theorem \ref{thmxred in neg reg}
is actually smooth.
\begin{theorem} \label{thmxhypoell}
If $\varrho = \varrho (\Phi )$ satisfies a Diophantine Condition with respect to $\a$,
then $\Phi $ is smoothly conjugate to (any representative of) $(\a , \exp (\varrho ))$.
\end{theorem}

In particular, we obtain the following affirmative answer to problem 2 of \cite{GrWal73}
in our context. We remind that the conjugation operator associated to a cocycle
$\Phi = (\a , A(\. ))$ is the operator acting on $\sm (\Td , G)$ by
\begeq
H(\. ) \mapsto H(\. + \a) .A(\. ) H^{*}(\. )
\endeq
Theorem \ref{thmxhypoell} is thus the non-linear analogue of the main theorem in
\cite{NKInvDist}.
\begin{corollary}
If the conjugation operator of an almost reducible cocycle (or quasi-periodic flow) in
$\Td \times G$, $G$ a compact Lie group, is Globally Hypoelliptic, then $G$ is a torus.
\end{corollary}

Some technicalities aside (that need to be settled), the same holds in the total space
$SW^{\infty}_{\a} (\T ,G)$ of one-frequency cocycles over $RDC$ rotations (see the
discussion preceeding eq. \eqref{eqskew shift}), and we expect the same to hold for
cocycles over all irrational rotations.

\textbf{Acknowledgement}: This work was supported by the LABEX CEMPI (ANR-11-LABX-0007-01).

\section{The basic lemmas}

The results proving the well-posedness of the definition of $\varrho$ are proved
via the following sequence of lemmas.
The reader not familiar with the structure of Lie groups can replace $G $ with
$G = SU(2)$ in order to keep algebra simple. As in \cite{NKRigidity}, this makes
the arguments more transparent without harming generality.

\begin{assumption}
In the lemmas throughout the section, we assume that the quantity $M$ of
theorem \ref{thm almost red} is infinite, in order to treat the difficult and interesting
case.
\end{assumption}

\begin{lemma} \label{lemdef rot vec}
Let $\Phi = (\a , A(\. ) )$ be an almost reducible cocycle. Then, there exists
$H_{i}(\. ) \in \sm (\Td ,G)$ a sequence of smooth conjugations converging
in $H^{-d-d'}$ and a sequence
$(\a , \L _{i})$ of resonant constant cocycles, $\L _{i} \in \TT $, such that
\begeq
\L _{i} \ra A _{\infty} \in \TT
\endeq
and, if we call $(\a ,A_{i}(\. )) = \Conj _{H_{i}(\.)} \Phi$, then
\begeq
\L _{i}^{*}A_{i}(\. ) \overset{\sm}{\ra} \Id
\endeq
\end{lemma}
This is the form of theorem
\ref{thm almost red} that we will use in the proof of theorem \ref{thmxhypoell}.
\begin{proof}
In the notation of \cite{NKContSpec}, \S 3.4.3, or theorem \ref{thm almost red},
assume that the cocycle is in K.A.M. normal
form. Then, the conjugations
\begeq
G'_{i}(\. ) = B_{i}^{*}(\. ) \exp (D_{i}) B_{i}(\. )
\endeq
with notation as in item 3, \S 5.1 op.cit.\footnote{There is a typo in the reference, the correct index
for the matrix $D_{i}$ is $i$ and not $i+1$.} and eq. \eqref{eqfar from id conj},
satisfy
\begeq
d_{s}(G'_{i}(\. ) , \Id ) \leq C_{s} |k_{i}|^{s} ,
k_{i} = (k_{i}^{\r })_{\r \in \D} \in \Z ^{\# \D}
\endeq
and they have only the resonant Fourier coefficients, precisely $k_{i}^{\r }$. This
sequence is summable in, say, $H^{-d-d'}$. This allows us to conclude, since the
conjugations
\begin{equation} \label{eqseq conj}
H_{i} (\. ) = \prod _{i=\infty}^{1}G'_{i}(\. )
\end{equation}
conjugate the cocycle $(\a , A(\. ))$ to\footnote{Notice the abuse of notation. We also
call the new sequence of cocycles $(\a , A_{i}(\. ))$.}
\begeq
\Conj _{H_{i}(\. )}\Phi = (\a , A_{i}e^{F_{i}(\. )}) =
(\a , A_{i} + O(|k_{i}|^{-\infty}))
\endeq
with
\begeq
d(\L _{i}^{*}A_{i}, \Id )
\leq |k_{i}|^{-\nu } \text{ and } \L_{i}, A_{i} \in \TT , \forall i, 1\leq i < M
\endeq
and
\begeq
d(A _{i}^{*}A_{i+1}, \Id )
= O( |k_{i}|^{-\infty })
\endeq
This ends the proof.
\end{proof}

\begin{lemma} \label{lemconj const}
Let $(\a ,A_{i})$, $j=1,2$ be two constant cocycles, $A_{i} \in \TT$, that are conjugate in
$\DD '$. Then, they are $\sm$ conjugate via a torus morphism $B(\. ): \Td \ra \TT$.
Moreover,
\begeq
A _{1}^{*}. A _{2} \in \exp ((\a \. \Zd)\ZZ )
\text{ and }
B( \Z ^{d} ) \subset (\ZZ )
\endeq
\end{lemma}
\begin{proof}
It is \S 5.1 of \cite{NKRigidity} where the measurability assumption on the
conjugation is unnecessary, as the Fourier transform is well defined for
distributions.
\end{proof}

\begin{lemma} \label{lemconv conj}
Let $H_{i}' (\. ) \in \sm (\Td ,G)$ converge in some $H^{s}$, and $\L _{i}' \in \TT $ be
a sequence of constants, both otherwise satisfying the same conclusions as
in lemma
\ref{lemdef rot vec} with $ A _{\infty}'$ replacing $ A _{\infty}$. Then
\begeq
A _{\infty}^{*}. A _{\infty}' \in \exp ((\a \. \Zd)\ZZ )
\endeq
and there exists a $\sm$ smooth conjugation $D(\. ) $ such that
\begeq
H_{i} (\.) = D(\. ) .H_{i}'(\. )
\endeq
In particular, $H_{i}' (\. )$ also converges in $H^{-d-d'}$.
\end{lemma}

\begin{proof}
The proof follows directly from \S 5.1 of \cite{NKContSpec}. Assume without loss of
generality that both sequences of conjugations $H_{i}(\. )$ and $H_{i}'(\. )$
correspond to K.A.M. normal forms of the cocycle, in which case the conjugations
$H_{i}(\. )$ and $H_{i}'(\. )$ are given by eq. \eqref{eqseq conj}.

If a product of admissible conjugations between K.A.M. normal forms converges in $\DD ' $,
then it contains a finite number of conjugations of the type denoted by
$B_{i}'(\. )$ as in eq. \eqref{eqfar from id conj}, for, otherwise, the product diverges in
$\DD '$. Products of all other admissible conjugations converge in $H^{-d-d'}$.

Since all other admissible conjugations do not modify the limit of the sequences
$\L_{i} \ra A_{\infty}$ and $\L _{i}' \ra A_{\infty}'$,
and the conjugation $D(\. )$ is given by the product of finitely many torus morphisms
$\Td \ra \TT$ and conjugations converging in $H^{-d-d'}$, the proof is complete.
\end{proof}

\begin{lemma} \label{lemdiv conj}
Let $H_{i} (\. )\in \sm (\Td ,G)$, $\L _{i} \in \TT $ and $A_{\infty}$ be as in lemma
\ref{lemdef rot vec}.
If $H_{i}' (\. )\in \sm (\Td ,G)$ and $\L _{i}' \in \TT $ with
$\L _{i}'  \ra A_{\infty}'$ satisfy definition \ref{defar}, but
\begin{equation} \label{eqdif notin lattice}
A _{\infty}^{*}. A _{\infty}' \notin \exp ((\a \. \Zd)\ZZ )
\end{equation}
then $H_{i}' (\. )$ does not converge in $\DD '$.
\end{lemma}

\begin{proof}
Follows directly from the proof of lemma \ref{lemconv conj}, since for eq.
\eqref{eqdif notin lattice} to hold, infinitely many conjugations of the type
$B_{i}'(\. )$ must occur, in which case the product does not converge in any distribution
space, and such conjugations have optimal Sobolev norms.
\end{proof}

\begin{lemma} \label{lemdiv csts}
If $H_{i}' (\. ) \in \sm (\Td ,G)$ and $\L _{i}' \in \TT $  satisfy definition \ref{defar},
except that $\L_{i}'$ does not converge, then $H_{i}'(\. )$ does not converge in $\DD '$.
\end{lemma}
\begin{proof}
The same line of argument as in the proof of lemma \ref{lemdiv conj}.
\end{proof}

\section{Existence and well-posedness of the rotation vector}

In this section, we provide the proofs of theorem \ref{thmxrot vec}, propositions
\ref{propxrot vec transl} and \ref{propxrot vec inv}, and theorem \ref{thmxred in neg reg}.
Taken altogether, they establish that the fibered rotation vector of an almost reducible
cocycle, cf. definition \ref{defrot vec}, is a well-defined object and a total invariant
of the dynamics, albeit in low regularity.

The well-posedness of the definition amounts to proving a weak form of rigidity for
the conjugations produced by the K.A.M. scheme, and to providing a classification of
the possible limits of the sequences of constants under this weak rigidity assumption.

\begin{proof}[Proof of theorem \ref{thmxrot vec}]
Lemma \ref{lemdef rot vec} proves the theorem, which in turn implies that definition
\ref{defrot vec} is not void.
\end{proof}

\begin{proof}[Proof of proposition \ref{propxrot vec transl}]
The proof follows by combining lemmas \ref{lemconv conj}, \ref{lemdiv conj}  and
\ref{lemdiv csts}
\end{proof}

\begin{proof}[Proof of proposition \ref{propxrot vec inv}]
The proof follows by combining lemmas \ref{lemconj const}, \ref{lemconv conj} and
\ref{lemdiv conj}.

The continuous $C^{0}$ dependence of representatives of $\varrho$ follows from the
K.A.M. normal form. Fix two almost reducible cocycles $\Phi ^{(j)}$ that are $C^{0}$ close,
and assume without loss of generality that $\Phi ^{(1)}$ is given in K.A.M. normal form.
Denote by $\varrho _{j}$ their respective rotation vectors.

The K.A.M. normal form of $\Phi ^{(1)}$ and the corresponding sequence of conjugations
as in eq. \eqref{eqseq conj} fixes a representative $a_{\infty}^{(1)}$ of
$\varrho _{1}$. Then, by following the K.A.M. scheme applied to $\Phi  ^{(2)}$, and by
constructing the corresponding sequence of conjugations as in eq. \eqref{eqseq conj} one
can obtain a representative $a_{\infty}^{(2)}$ of $\varrho _{2}$ whose distance to
$\varrho _{1}$ is bounded by the $C^{0} $ distance of the cocycles.
\end{proof}

\begin{proof} [Proof of theorem \ref{thmxred in neg reg}]
It follows directly from lemma \ref{lemdef rot vec}.
\end{proof}

\section{Proof of theorem \ref{thmxhypoell}}

The proof of theorem 1.1 of \cite{NKRigidity}, applies verbatim. This is so, because the
only function of the measurable conjugation is to assure that the K.A.M. scheme produces a
finite number of resonances. This happens when $ a _{\infty}$ ($a _{\infty}$ corresponds
to $a_{d}$ of the reference) is polynomially away from resonances.

%

\begin{proof}[Proof of theorem \ref{thmxhypoell}]
Fix $a_{\infty}$, the representative of $\varrho (\Phi)$ obtained as the limit of cocycles
constructed in the proof of lemma \ref{lemdef rot vec} assuming that $M$, the number of
resonant steps, is infinite. Then, by assumption, $A_{\infty} = \exp (a_{\infty})$
satisfies
\begeq
d(\L _{i}^{*}.A_{\infty}, \Id ) \gtrsim |k_{i}|^{-\tilde{\t }}
\endeq
We also have by construction that
\begeq
d(\L _{i}^{*}.A_{i}, \Id ) < |k_{i}|^{-\nu }
\endeq
Assume, now, without loss of generality (cf. \cite{NKRigidity} for the same argument)
that $\nu > \tilde{\t}$. Then, we immediately get
\begeq
d(A _{i}^{*}.A_{\infty}, \Id ) \gtrsim |k_{i}|^{-\tilde{\t }}
\endeq
On the other hand, by construction
\begeq
d(A _{i+1}^{*}.A_{i}, \Id ) = O(|k_{i}|^{-\infty })
\endeq
which implies that
\begeq
d(A _{i}^{*}.A_{\infty}, \Id ) = O(|k_{i}|^{-\infty })
\endeq
This is a contradiction, which forces $M$ to be finite so that $\Phi$ is $\sm$ reducible.
\end{proof}

As we have already remarked, the Diophantine condition on the fibered rotation vector
$\varrho$ is optimal for the hypoellipticity of the conjugation operator.

\section{Some comments}

Our proofs of the results needed in order to justify the definition of the rotation
vector $\varrho$ rely heavily on the K.A.M. normal form, which is available only for
almost reducible cocycles over Diophantine rotations. We expect, however, the rotation
vector to be defined also for almost reducible cocycles over Liouvillean rotations.

The original definition of a rotation number was given by H. Poincar\'{e} for
homeomorphisms of $\T ^{1}$ (see, e.g. \cite{KatHass}), who also proved that a circle
diffeomorphism with an irrational rotation number is semi-conjugate to the corresponding
rotation. Subsequently, A. Denjoy proved that a twice differentiable diffeomorphism is
actually conjugate to the rotation. The relevance of the arithmetics of the rotation
number in questions of smoothness of the conjugation was showed by V.I. Arnol'd in
\cite{Arnold61}, where he proved that real analytic perturbations of a Diophantine
rotation $\a \in \T$ whose rotation number is $\a$ are analytically conjugate to the
rotation. The theory was pushed further by M. Herman and J.-C. Yoccoz, who established
that the rotation number is a total invariant for smooth diffeomorphisms of the circle
(cf. \cite{YocAst}).

Diffeomorphisms of tori of higher dimension can be assigned rotation sets rather than
vectors, see \cite{MisZiem91} for the definition. The analog of Arnol'd's theorem
for higher dimensional tori is proved in \cite{NKKAMTor}.

In the class of fibered dynamics, fibered rotation numbers can be defined for
quasi-periodically forced circle maps. These are diffeomorphisms of $\T ^{2}$ that are
homotopic to the $\Id $ and of the form
\begeq
(x , y ) \mapsto (x+\a , g(x,y))
\endeq
where for each $x \in \T$ the mapping $y \mapsto g(x,y)$ is a diffeomorphism of $\T ^{1}$
In this case, the arithmetics of the rotation number in the $y$ component, the fibered
rotation number, play an important role in the regularity of the reducing
conjugation (see, e.g. \cite{JagStark06} for a discussion).

For the same reason as for quasi-periodically forced circle maps, quasi-periodic cocycles
in $\Td \times SL(2,\R )$ that are homotopic to the $\Id$ can be assigned a fibered
rotation number by projectivization in the $SL(2, \R )$ direction. The importance of the
aritmetics of the fibered rotation number with respect to the rotation in the basis are
exhibited by results like the one obtained by H. Eliasson in \cite{El1992}.

For all of the above systems, the rotation number or set, fibered or not, is defined
by dynamics alone. However, in the context of this note we do not dispose of combinatorial
arguments, or, at least a priori, of any torus where the dynamics naturally lives.

The rotation vector can nonetheless be defined, because we manage to conjugate any given
cocycle close to a converging sequence of constant cocycles, all belonging to the same
torus, and without losing too much control of the sequence of conjugations.
We think that this should also be the case when the rotation in the basis is Liouvillean,
in which case a rotation vector would be defined for an open set of cocycles over all
minimal rotations.

\bibliography{nikosbib}
\bibliographystyle{aomalpha}

\end{document}